\documentclass{amsart}
\usepackage{amsfonts}
\usepackage{amsmath}
\usepackage{hyperref}

\newtheorem{theorem}{Theorem}
\theoremstyle{plain}

\newtheorem{corollary}{Corollary}

\newtheorem{remark}{Remark}

\numberwithin{equation}{section}

\input{tcilatex}
\begin{document}

\title[Two point Gauss--Legendre Quadrature Rule
for $\mathcal{RS}$--integral] {Two point Gauss--Legendre
Quadrature Rule for Riemann--Stieltjes integrals}
\author[M.W. Alomari]{Mohammad W. Alomari}
\address{Department of Mathematics,
Faculty of Science and Information Technology, Jadara University, 21110 Irbid, Jordan}
\email{mwomath@gmail.com}

%\date{Dec 19, 2012.}
\date{\today}


\begin{abstract}
In order to approximate the Riemann--Stieltjes integral $\int_a^b
{f\left( t \right)dg\left( t \right)}$ by $2$--point Gaussian
quadrature rule, we introduce the quadrature rule
\begin{align*}
\int_{ - 1}^1 {f\left( t \right)dg\left( t \right)}  \approx A
f\left( { - \frac{{\sqrt 3 }}{3}} \right) + B f\left(
{\frac{{\sqrt 3 }}{3}} \right),
\end{align*}
for suitable choice of $A$ and $B$. Error estimates for this
approximation under various assumptions for the functions involved
are provided as well.
\end{abstract}

\maketitle

%=============================================================================
\section{Introduction}
%=============================================================================
In numerical analysis, inequalities play a main role in error
estimations. A few years ago, by using modern theory of
inequalities and Peano kernel approach a number of authors have
considered an error analysis of some quadrature rules of
Newton-Cotes type. In particular, the Mid-point, Trapezoid and
Simpson's rules have been investigated more recently with the view
of obtaining bounds for the quadrature rules in terms of at most
first derivative. A comprehensive list of preprints related to
this subject may be found at http://ajmaa.org/RGMIA

The Newton--Cotes formulas use values of function at equally
spaced points. The same practice when the formulas are combined to
form the composite rules, but this restriction can significantly
decrease the accuracy of the approximation. In fact, these methods
are inappropriate when integrating a function on an interval that
contains both regions with large functional variation and regions
with small functional variation. If the approximation error is to
be evenly distributed, a smaller step size is needed for the
large-variation regions than for those with less variation.

Among others Gaussian quadrature rules gives the highest possible
degree of precision; so that it is recommended to be `almost' the
method of choice.

In order to investigate $2$-point Gauss-Legendre quadrature rule,
Ujevi\'{c} \cite{Ujevic1} obtained bounds for absolutely
continuous functions with derivatives belong to $L_2(a,b)$, as
follows:
\begin{theorem}
Let $f:[-1,1] \longrightarrow \mathbb{R}$ be an absolutely
continuous whose derivative $f^{\prime} \in L_2(-1,1)$. Then
\begin{align}
\label{eq1.1}\left| {f\left( {- \frac{\sqrt {3}}{3} } \right) +
f\left( {\frac{\sqrt {3}}{3}} \right) - \int_{ - 1}^1 {f\left( t
\right)dt} } \right| \le \sqrt {\frac{{4 - 2\sqrt 3 }}{3}}  \cdot
\sigma ^{1/2} \left( {f'} \right),
\end{align}
where,
\begin{align*}
\sigma \left( g \right) = \left( {b - a} \right)\mathcal{T}\left(
{g,g} \right);
\end{align*}
and
\begin{align*}
\mathcal{T}\left( {g,g} \right) = \frac{1}{{b - a}}\left\| g
\right\|_2^2  - \frac{1}{{\left( {b - a} \right)^2 }}\left(
{\int_a^b {g\left( t \right)dt} } \right)^2.
\end{align*}
Inequality (\ref{eq1.1}) is sharp in the sense that the constant
$\sqrt {\frac{{4 - 2\sqrt 3 }}{3}}$ cannot be replaced by a
smaller one.
\end{theorem}
\noindent Some Gaussian and Gaussian-like quadrature rules are
considered in \cite{Ujevic2}.
\newline

The Riemann--Stieltjes integral $\int_{a}^{b}{f\left( t\right)
dg\left( t\right) }$ is an important concept in Mathematics with
multiple applications in several subfields including Probability
Theory \& Statistics, Complex Analysis, Functional Analysis,
Operator Theory and others.

In Numerical Integration, the number of proposed quadrature rules
to approximate this type of integrals is very small by comparison
with the huge
number of methods available to approximate the classical Riemann integral $%
\int_{a}^{b}{f\left( t\right) dt}$.

In recent years, the approximation problem of the
Riemann--Stieltjes integral $\int_a^b {fdg}$ has been studied with
the tools of modern inequalities; therefore several error
approximation for proposed quadrature rules had been established.
The most famous and interesting approximations have been done
using Ostrowski and generalized trapezoid type inequalities.

In 2000, Dragomir \cite{Dragomir3} (see also \cite{Dragomir2}) has
introduced the following significant quadrature rule:
\begin{align*}
 \int_a^b {f\left( t \right)du\left( t \right)} \cong f\left( x \right) \left[ {u\left( {b} \right) - u\left( a \right)}
\right],
\end{align*}
and several error bounds under various assumptions to the function
involved were obtained, the reader may refer to \cite{alomari3},
\cite{Barnett1}--\cite{CeroneDragomir1},
\cite{Dragomir2}--\cite{Dragomir3} and the references therein.

From a different point of view, the authors of \cite{Dragomir4}
considered the problem of approximating the Riemann--Stieltjes
integral $\int_a^b {f\left( t \right)du\left( t \right)}$ via the
generalized trapezoid rule $\left[ {u\left( x \right) - u\left( a
\right)} \right]f\left( a \right) + \left[ {u\left( b \right) -
u\left( x \right)} \right]f\left( b \right)$, i.e.,
\begin{align*}
\int_a^b {f\left( t \right)du\left( t \right)} \cong \left[
{u\left( x \right) - u\left( a \right)} \right]f\left( a \right) +
\left[ {u\left( b \right) - u\left( x \right)} \right]f\left( b
\right), \forall x \in [a,b].
\end{align*}
For other related results see \cite{Barnett1,Barnett2} and
\cite{Dragomir5}--\cite{DragomirFedotov}.

In 2008, Mercer \cite{Mercer} proved a new version of
Hermite--Hadamard inequality for Riemann--Stieltjes integral, and
introduced the following Trapezoid type rule:
\begin{align}
\int_{a}^{b}{fdg}\approx \left[ {G-g\left( a\right) }\right] f\left( a\right) +%
\left[ {g\left( b\right) -G}\right] f\left( b\right)
\end{align}
where, $G:= \frac{1}{b - a} \int_a^b {g(t)dt}$.

More recently, Alomari \cite{alomari2} and \cite{alomari3}
introduced the quadrature rule:
\begin{align*}
\int_a^b {f\left( t \right)du\left( t \right)} \cong \left[
{u\left( {\frac{{a + b}}{2}} \right) - u\left( a \right)}
\right]f\left( x \right) + \left[ {u\left( b \right) - u\left(
{\frac{{a + b}}{2}} \right)} \right]f\left( {a + b - x} \right),
\end{align*}
$\forall x \in \left[ {a,\frac{{a + b}}{2}} \right]$, and
therefore several error bounds of this approximation rule under
various assumptions for the functions involved are proved. For new
result regarding the above quadrature rules the reader may refer
to \cite{alomari1} and \cite{alomari4}.

\section{Two--point Gaussian Quadrature Rule}

To establish a two point Gauss-Legendre quadrature rule for the
Riemann--Stieltjes integral $\int_a^b {f\left( t \right)dg\left( t
\right)} $, let us seek numbers $A$ and $B$ such that
\begin{align}
\label{eq2.1}\int_{ - 1}^1 {f\left( t \right)dg\left( t \right)}
\approx Af\left( { - \frac{{\sqrt 3 }}{3}} \right) + Bf\left(
{\frac{{\sqrt 3 }}{3}} \right).
\end{align}
To find the scalars $A$ and $B$, let $f(t)=1$ and then $f(t)=t$ in
(\ref{eq2.1}); respectively. By simple calculations we get
\begin{align}
A &= \frac{3}{{2\sqrt 3 }}\left[ {\int_{ - 1}^1 {g\left( t
\right)dt}  - \left( {\frac{3 - \sqrt 3}{3} } \right)g\left( 1
\right) - \left( {\frac{\sqrt 3  + 3}{3}} \right)g\left( { - 1}
\right)} \right]\label{eq2.2}
\end{align}
and
\begin{align}
B &=  \frac{3}{{2\sqrt 3 }}\left[ {\left( {\frac{3 + \sqrt 3}{3} }
\right)g\left( 1 \right) + \left( {\frac{3 - \sqrt 3}{3} }
\right)g\left( { - 1} \right) - \int_{ - 1}^1 {g\left( t
\right)dt} } \right],\label{eq2.3}
\end{align}
and therefore by (\ref{eq2.1}), we may write
\begin{multline}
\label{eq2.4}\int_{ - 1}^1 {f\left( t \right)dg\left( t \right)}
\\
\approx \frac{3}{{2\sqrt 3 }}\left[ {\int_{ - 1}^1 {g\left( t
\right)dt}  - \left( {\frac{3 - \sqrt 3}{3} } \right)g\left( 1
\right) - \left( {\frac{\sqrt 3  + 3}{3}} \right)g\left( { - 1}
\right)} \right]f\left( { - \frac{{\sqrt 3 }}{3}} \right)
\\
+\frac{3}{{2\sqrt 3 }}\left[ {\left( {\frac{3 + \sqrt 3}{3} }
\right)g\left( 1 \right) + \left( {\frac{3 - \sqrt 3}{3} }
\right)g\left( { - 1} \right) - \int_{ - 1}^1 {g\left( t
\right)dt} } \right]f\left( {\frac{{\sqrt 3 }}{3}} \right).
\end{multline}
As special cases, we have
\begin{enumerate}
\item If $g$ is an odd function, so that we have $ g\left( { - x}
\right) =  - g\left( x \right)$, for all $x \in [-1,1]$ and
$\int_{ - 1}^1 {g\left( t \right)dt} = 0$. Therefore,
(\ref{eq2.4}), becomes
\begin{align}
\label{eq2.5}\int_{ - 1}^1 {f\left( t \right)dg\left( t \right)}
&\approx g\left( { 1} \right) \left[ {f\left( { - \frac{{\sqrt 3
}}{3}} \right) + f\left( {\frac{{\sqrt 3 }}{3}} \right)} \right].
\end{align}
For instance, if $g(t):= t$ for all $t \in [-1,1]$, then
\begin{align}
\label{eq2.6}\int_{ - 1}^1 {f\left( t \right)dt}  \approx f\left(
{ - \frac{{\sqrt 3 }}{3}} \right) + f\left( {\frac{{\sqrt 3 }}{3}}
\right),
\end{align}
which reduces to the classical Gauss--Legendre quadrature formula
for the Riemann integral $\int_{ - 1}^1 {f\left( t \right)dt}$.

\item If $g$ is an even function, so that we have $ g\left( { - x}
\right) =  g\left( x \right)$, for all $x \in [-1,1]$ and $\int_{
- 1}^1 {g\left( t \right)dt} = 2 \int_{0}^1 {g\left( t
\right)dt}$. Therefore, (\ref{eq2.4}), becomes
\begin{align}
\label{eq2.7}\int_{ - 1}^1 {f\left( t \right)dg\left( t \right)}
&\approx \frac{3}{{\sqrt 3 }}\left[ {g\left( 1 \right) -
\int_{0}^1 {g\left( t \right)dt}} \right] \left[ {f\left(
{\frac{{\sqrt 3 }}{3}} \right) - f\left( { - \frac{{\sqrt 3 }}{3}}
\right) } \right].
\end{align}
\end{enumerate}

Now, we are ready to state our first result.

\begin{theorem}
\label{thm1}Let $f,g : [-1,1] \to \mathbb{R}$ be such that $f$ is
of $r$-$H_f$--Holder type on $[-1,1]$, i.e.,
\begin{align*}
\left| {f\left( x \right) - f\left( y \right)} \right| \le H_f\left| {x - y} \right|^r
\end{align*}
for all $x,y \in [-1,1]$, where $r>0$ and $H_f >0$
are given,
and $g$ is of bounded variation on $[-1,1]$. Then,
\begin{align}
\label{main.ineq.1}\left| {\int_{ - 1}^1 {f\left( t
\right)dg\left( t \right)} - Af\left( { - \frac{{\sqrt 3 }}{3}}
\right) - Bf\left( {\frac{{\sqrt 3 }}{3}} \right)} \right| \le H_f
\left( {\frac{{3 + \sqrt 3 }}{3}} \right)^r  \cdot \bigvee_{ -
1}^1 \left( g \right),
\end{align}
where $A$ and $B$ are given in (\ref{eq2.2}) and (\ref{eq2.3});
respectively.
\end{theorem}

\begin{proof}
From (\ref{eq2.2}) and (\ref{eq2.3}) it is easy to observe that
$A+B = g\left( 1 \right) - g\left( { - 1} \right)$. So that,
\begin{align*}
\mathcal{ER}\left( {f,g} \right)&:= \int_{ - 1}^1 {f\left( t
\right)dg\left( t \right)}  - Af\left( { - \frac{{\sqrt 3 }}{3}}
\right) - Bf\left( {\frac{{\sqrt 3 }}{3}} \right)
\\
&= \int_{ - 1}^1 {f\left( t \right)dg\left( t \right)}  -
\frac{1}{{g\left( 1 \right) - g\left( { - 1} \right)}}\int_{ -
1}^1 {\left[ {Af\left( { - \frac{{\sqrt 3 }}{3}} \right) +
Bf\left( {\frac{{\sqrt 3 }}{3}} \right)} \right]dg\left( t
\right)}
\\
&= \int_{ - 1}^1 {\left[ {f\left( t \right) - \frac{{Af\left( { -
{{\sqrt 3 } \mathord{\left/ {\vphantom {{\sqrt 3 } 3}} \right.
\kern-\nulldelimiterspace} 3}} \right) + Bf\left( {{{\sqrt 3 }
\mathord{\left/ {\vphantom {{\sqrt 3 } 3}} \right.
\kern-\nulldelimiterspace} 3}} \right)}}{{g\left( 1 \right) -
g\left( { - 1} \right)}}} \right]dg\left( t \right)}.
\end{align*}
It is well-known that for a continuous function $p:[a,b] \to
\mathbb{R}$ and a function $\nu:[a,b] \to \mathbb{R}$ of bounded
variation, one has the inequality
\begin{align}
\label{eq2.8}\left| {\int_a^b {p\left( t \right)d\nu\left( t
\right)} } \right| \le \mathop {\sup }\limits_{t \in \left[ {a,b}
\right]} \left| {p\left( t \right)} \right| \bigvee_a^b\left( \nu
\right).
\end{align}
Using (\ref{eq2.8}), we have
\begin{align*}
\left| {\mathcal{ER}\left( {f,g} \right)} \right| &= \left|
{\int_{ - 1}^1 {\left[ {f\left( t \right) - \frac{{Af\left( { -
{{\sqrt 3 } \mathord{\left/
 {\vphantom {{\sqrt 3 } 3}} \right.
 \kern-\nulldelimiterspace} 3}} \right) + Bf\left( {{{\sqrt 3 } \mathord{\left/
 {\vphantom {{\sqrt 3 } 3}} \right.
 \kern-\nulldelimiterspace} 3}} \right)}}{{g\left( 1 \right) - g\left( { - 1} \right)}}} \right]dg\left( t \right)} } \right|
 \\
&\le \sup \left| {f\left( t \right) - \frac{{Af\left( { - {{\sqrt
3 } \mathord{\left/
 {\vphantom {{\sqrt 3 } 3}} \right.
 \kern-\nulldelimiterspace} 3}} \right) + Bf\left( {{{\sqrt 3 } \mathord{\left/
 {\vphantom {{\sqrt 3 } 3}} \right.
 \kern-\nulldelimiterspace} 3}} \right)}}{{g\left( 1 \right) - g\left( { - 1} \right)}}} \right| \cdot \bigvee_{ - 1}^1 \left( g \right)
\\
&= \frac{1}{{g\left( 1 \right) - g\left( { - 1} \right)}}\sup
\left| {\left[ {g\left( 1 \right) - g\left( { - 1} \right)}
\right]f\left( t \right) - \left[ {Af\left( { - {{\sqrt 3 }
\mathord{\left/
 {\vphantom {{\sqrt 3 } 3}} \right.
 \kern-\nulldelimiterspace} 3}} \right) + Bf\left( {{{\sqrt 3 } \mathord{\left/
 {\vphantom {{\sqrt 3 } 3}} \right.
 \kern-\nulldelimiterspace} 3}} \right)} \right]} \right| \cdot \bigvee_{ - 1}^1 \left( g \right)
\\
&= \frac{1}{{g\left( 1 \right) - g\left( { - 1} \right)}}\sup
\left| {\left[ {A + B} \right]f\left( t \right) - \left[ {Af\left(
{ - {{\sqrt 3 } \mathord{\left/
 {\vphantom {{\sqrt 3 } 3}} \right.
 \kern-\nulldelimiterspace} 3}} \right) + Bf\left( {{{\sqrt 3 } \mathord{\left/
 {\vphantom {{\sqrt 3 } 3}} \right.
 \kern-\nulldelimiterspace} 3}} \right)} \right]} \right| \cdot \bigvee_{ - 1}^1 \left( g \right)
\\
&= \frac{1}{{g\left( 1 \right) - g\left( { - 1} \right)}}\sup
\left| {Af\left( t \right) + Bf\left( t \right) - Af\left( { -
{{\sqrt 3 } \mathord{\left/
 {\vphantom {{\sqrt 3 } 3}} \right.
 \kern-\nulldelimiterspace} 3}} \right) - Bf\left( {{{\sqrt 3 } \mathord{\left/
 {\vphantom {{\sqrt 3 } 3}} \right.
 \kern-\nulldelimiterspace} 3}} \right)} \right| \cdot \bigvee_{ - 1}^1 \left( g \right)
\\
&\le \frac{1}{{g\left( 1 \right) - g\left( { - 1} \right)}}\left[
{A\sup \left| {f\left( t \right) - f\left( { - {{\sqrt 3 }
\mathord{\left/
 {\vphantom {{\sqrt 3 } 3}} \right.
 \kern-\nulldelimiterspace} 3}} \right)} \right| + B\sup \left| {f\left( t \right) - f\left( {{{\sqrt 3 } \mathord{\left/
 {\vphantom {{\sqrt 3 } 3}} \right.
 \kern-\nulldelimiterspace} 3}} \right)} \right|} \right] \cdot \bigvee_{ - 1}^1 \left( g \right)
\\
&\le \frac{H_f}{{g\left( 1 \right) - g\left( { - 1}
\right)}}\left[ {A \mathop {\sup }\limits_{t \in \left[ { - 1,1}
\right]} \left| {t + \frac{1}{{\sqrt 3 }}} \right|^r  + B \mathop
{\sup }\limits_{t \in \left[ { - 1,1} \right]} \left| {t -
\frac{1}{{\sqrt 3 }}} \right|^r } \right] \cdot \bigvee_{ - 1}^1
\left( g \right)
\end{align*}
\begin{align*}
&= \frac{H_f}{{g\left( 1 \right) - g\left( { - 1} \right)}} \left(
{A + B} \right)\left( {{\frac{{3 + \sqrt 3 }}{3}}} \right)^r \cdot
\bigvee_{ - 1}^1 \left( g \right)
\\
&= H_f \left( {{\frac{{3 + \sqrt 3 }}{3}}} \right)^r  \cdot
\bigvee_{ - 1}^1 \left( g \right),
\end{align*}
which gives the required result.
\end{proof}

\begin{corollary}
Let $f : [-1,1] \to \mathbb{R}$ be such that $f$ is of
$r$-$H_f$--Holder type on $[-1,1]$, where $r>0$ and $H_f >0$ are
given. Then,
\begin{align}
\left| {\int_{ - 1}^1 {f\left( t \right)dt} - f\left( { -
\frac{{\sqrt 3 }}{3}} \right) - f\left( {\frac{{\sqrt 3 }}{3}}
\right)} \right| \le 2 H_f \left( {\frac{{3 + \sqrt 3 }}{3}}
\right)^r.
\end{align}
\end{corollary}

\begin{theorem}
\label{thm2}Let $f,g : [-1,1] \to \mathbb{R}$ be such that $f$ is
of $r$-$H_f$--Holder type on $[-1,1]$, where $r>0$ and $H_f >0$
are given, and $g$ is $L_g$--Lipschitzian on $[-1,1]$. Then,
\begin{multline}
\label{main.ineq.2}\left| {\int_{ - 1}^1 {f\left( t
\right)dg\left( t \right)} - Af\left( { - \frac{{\sqrt 3 }}{3}}
\right) - Bf\left( {\frac{{\sqrt 3 }}{3}} \right)} \right|
\\
\le \frac{{L_g H_f }}{{r + 1}}\left[ {\left( {\frac{{3 - \sqrt 3
}}{3}} \right)^{r + 1}  + \left( {\frac{{3 + \sqrt 3 }}{3}}
\right)^{r + 1} } \right],
\end{multline}
where $A$ and $B$ are given in (\ref{eq2.2}) and (\ref{eq2.3});
respectively.
\end{theorem}

\begin{proof}
It is well-known that for a Riemann integrable function $p:[a,b]
\to \mathbb{R}$ and $L$--Lipschitzian function $\nu:[a,b] \to
\mathbb{R}$, one has the inequality
\begin{align}
\left| {\int_a^b {p\left( t \right)d\nu\left( t \right)} } \right|
\le L \int_a^b{\left| {p\left( {t} \right)} \right|dt}.
\label{eq2.10}
\end{align}
Using (\ref{eq2.10}), we have
\begin{align*}
\left| {\mathcal{ER}\left( {f,g} \right)} \right| &= \left|
{\int_{ - 1}^1 {\left[ {f\left( t \right) - \frac{{Af\left( { -
{{\sqrt 3 } \mathord{\left/
 {\vphantom {{\sqrt 3 } 3}} \right.
 \kern-\nulldelimiterspace} 3}} \right) + Bf\left( {{{\sqrt 3 } \mathord{\left/
 {\vphantom {{\sqrt 3 } 3}} \right.
 \kern-\nulldelimiterspace} 3}} \right)}}{{g\left( 1 \right) - g\left( { - 1} \right)}}} \right]dg\left( t \right)} } \right|
 \\
&\le L_g \int_{ - 1}^1 {\left| {f\left( t \right) -
\frac{{Af\left( { - {{\sqrt 3 } \mathord{\left/
 {\vphantom {{\sqrt 3 } 3}} \right.
 \kern-\nulldelimiterspace} 3}} \right) + Bf\left( {{{\sqrt 3 } \mathord{\left/
 {\vphantom {{\sqrt 3 } 3}} \right.
 \kern-\nulldelimiterspace} 3}} \right)}}{{g\left( 1 \right) - g\left( { - 1} \right)}}} \right|dt}
\\
&= \frac{{L_g }}{{g\left( 1 \right) - g\left( { - 1}
\right)}}\int_{ - 1}^1 {\left| {\left[ {g\left( 1 \right) -
g\left( { - 1} \right)} \right]f\left( t \right) - \left[
{Af\left( { - {{\sqrt 3 } \mathord{\left/
 {\vphantom {{\sqrt 3 } 3}} \right.
 \kern-\nulldelimiterspace} 3}} \right) + Bf\left( {{{\sqrt 3 } \mathord{\left/
 {\vphantom {{\sqrt 3 } 3}} \right.
 \kern-\nulldelimiterspace} 3}} \right)} \right]} \right|dt}
\\
&= \frac{{L_g }}{{g\left( 1 \right) - g\left( { - 1}
\right)}}\int_{ - 1}^1 {\left| {\left[ {A + B} \right]f\left( t
\right) - \left[ {Af\left( { - {{\sqrt 3 } \mathord{\left/
 {\vphantom {{\sqrt 3 } 3}} \right.
 \kern-\nulldelimiterspace} 3}} \right) + Bf\left( {{{\sqrt 3 } \mathord{\left/
 {\vphantom {{\sqrt 3 } 3}} \right.
 \kern-\nulldelimiterspace} 3}} \right)} \right]} \right|dt}
\\
&= \frac{{L_g }}{{g\left( 1 \right) - g\left( { - 1}
\right)}}\int_{ - 1}^1 {\left[ {A\left| {f\left( t \right) -
f\left( { - {{\sqrt 3 } \mathord{\left/
 {\vphantom {{\sqrt 3 } 3}} \right.
 \kern-\nulldelimiterspace} 3}} \right)} \right| + B\left| {f\left( t \right) - f\left( {{{\sqrt 3 } \mathord{\left/
 {\vphantom {{\sqrt 3 } 3}} \right.
 \kern-\nulldelimiterspace} 3}} \right)} \right|} \right]dt}
\\
&= \frac{{L_g H_f }}{{g\left( 1 \right) - g\left( { - 1}
\right)}}\int_{ - 1}^1 {\left[ {A\left| {t + \frac{{\sqrt 3 }}{3}}
\right|^r  + B\left| {t - \frac{{\sqrt 3 }}{3}} \right|^r }
\right]dt}
\\
&= \frac{{L_g H_f }}{{r + 1}}\left[ {\left( {\frac{{3 - \sqrt 3
}}{3}} \right)^{r + 1}  + \left( {\frac{{3 + \sqrt 3 }}{3}}
\right)^{r + 1} } \right],
\end{align*}
which gives the required result.
\end{proof}

\begin{corollary}
Let $g$ be as in Theorem \ref{thm2}. If $f$ is $L_f$--Lipschitzian
on $[-1,1]$. Then,
\begin{align}
\label{main.ineq.2}\left| {\int_{ - 1}^1 {f\left( t
\right)dg\left( t \right)} - Af\left( { - \frac{{\sqrt 3 }}{3}}
\right) - Bf\left( {\frac{{\sqrt 3 }}{3}} \right)} \right| \le
\frac{4}{{3}} L_f L_g,
\end{align}
where $A$ and $B$ are given in (\ref{eq2.2}) and (\ref{eq2.3});
respectively.
\end{corollary}

\begin{corollary}
Let $f$ be as in Theorem \ref{thm2}. Take $g(t)= t$ on $[-1,1]$.
Then,
\begin{multline}
\left| {\int_{ - 1}^1 {f\left( t \right)dt} - f\left( { -
\frac{{\sqrt 3 }}{3}} \right) - f\left( {\frac{{\sqrt 3 }}{3}}
\right)} \right|
\\
\le \frac{{H_f }}{{r + 1}}\left[ {\left( {\frac{{3 - \sqrt 3
}}{3}} \right)^{r + 1}  + \left( {\frac{{3 + \sqrt 3 }}{3}}
\right)^{r + 1} } \right].
\end{multline}
\end{corollary}

\begin{remark}
\begin{enumerate}
\item[a)] A result for a monotonic non-decreasing integrator may
be stated by using the fact that for a monotonic non-decreasing
function $\nu:[a,b] \to \mathbb{R}$ and continuous function
$p:[a,b] \to \mathbb{R}$, one has the inequality
\begin{align*}
\left| {\int_a^b {p\left( t \right)d\nu\left( t \right)} } \right|
\le \int_a^b {\left| {p\left( t \right)} \right|d\nu \left( t
\right)}.
\end{align*}

\item[b)] Another result(s) in terms of $L_p$ norms may be stated
by applying the well--known H\"{o}lder integral inequality, by
noting that
\begin{align*}
\left| {\int_c^d {h\left( s \right)du\left( s \right)} } \right|
\le \sqrt[q]{{u\left( d \right) - u\left( c \right)}} \times
\sqrt[p]{{\int_c^d {\left| {h\left( s \right)} \right|^p du\left(
s \right)} }}.
\end{align*}
where, $p>1$, $\frac{1}{p} + \frac{1}{q} = 1$.
\end{enumerate}
\end{remark}

A result for absolutely continuous integrand whose derivative
belongs to $L_2(-1,1)$ is given as follows:
\begin{theorem}
Let $f,g:[-1,1] \longrightarrow \mathbb{R}$ be two absolutely
continuous functions whose derivatives $f^{\prime}, g^{\prime} \in
L_2(-1,1)$ and $\int_{ - 1}^1 {f\left( t \right)g'\left( t
\right)dt}$ exists. Then
\begin{align}
\left| {\int_{ - 1}^1 {f\left( t \right)g'\left( t \right)dt} -
Af\left( { - \frac{{\sqrt 3 }}{3}} \right) - Bf\left(
{\frac{{\sqrt 3 }}{3}} \right)} \right| \le  \sigma ^{1/2} \left(
f \right) \cdot \sigma ^{1/2} \left( {g'} \right), \label{eq2.14}
\end{align}
where $A$ and $B$ are given in (\ref{eq2.2}) and (\ref{eq2.3});
respectively.
\begin{align*}
\sigma \left( h \right) = 2\mathcal{T}\left( {h,h} \right);
\end{align*}
and
\begin{align*}
\mathcal{T}\left( {h,h} \right) = \frac{1}{{2}}\left\| h
\right\|_2^2  - \frac{1}{{4}}\left( {\int_{-1}^{1} {h\left( t
\right)dt} } \right)^2.
\end{align*}
\end{theorem}
\begin{proof}
Since $A+B = g\left( 1 \right) - g\left( { - 1} \right)$. For
simplicity, we set
\begin{equation*}
F = {\frac{{Af\left( { - 1/\sqrt 3 } \right) + Bf\left( {1/\sqrt
3} \right)}}{{A + B}}}.
\end{equation*}
Define the mapping $K(t)= f(t) - F$. It is easy to observe that
\begin{align}
\mathcal{ER}\left( {f,g} \right) &=\int_{ - 1}^1 {K\left( t
\right)g'\left( t \right)dt}
\nonumber\\
&= \int_{ - 1}^1 {\left[ {K\left( t \right) - \frac{1}{2}\int_{ -
1}^1 {K\left( s \right)ds} } \right]\left[ {g'\left( t \right) -
\frac{1}{2}\int_{ - 1}^1 {g'\left( s \right)ds} } \right]dt}
\nonumber\\
&=2\mathcal{T}\left( {K,g'} \right). \label{eq2.15}
\end{align}
So that, we may write
\begin{align}
\mathcal{T}^2\left( {K,g'} \right) &= \frac{1}{4}\left\{ {\int_{ -
1}^1 {\left[ {K\left( t \right) - \frac{1}{2}\int_{ - 1}^1
{K\left( s \right)ds} } \right]\left[ {g'\left( t \right) -
\frac{1}{2}\int_{ - 1}^1 {g'\left( s \right)ds} } \right]dt}}
\right\}^2
\nonumber\\
&\le \frac{1}{4}\int_{ - 1}^1 {\left[ {K\left( t \right) -
\frac{1}{2}\int_{ - 1}^1 {K\left( s \right)ds} } \right]^2 dt}
\cdot \int_{ - 1}^1 {\left[ {g'\left( t \right) -
\frac{1}{2}\int_{ - 1}^1 {g'\left( s \right)ds} } \right]^2
dt}\label{eq2.16}
\end{align}
But since
\begin{align*}
\int_{ - 1}^1 {\left[ {K\left( t \right) - \frac{1}{2}\int_{ -
1}^1 {K\left( s \right)ds} } \right]^2 dt} &= \int_{ - 1}^1
{\left[ {\left( {f\left( t \right) - F} \right) -
\frac{1}{2}\int_{ - 1}^1 {\left( {f\left( s \right) - F}
\right)ds} } \right]^2 dt}
\\
&= \int_{ - 1}^1 \left[ {\left( {f\left( t \right) -
\frac{{Af\left( { - 1/\sqrt 3 } \right) + Bf\left( {1/\sqrt 3 }
\right)}}{{A + B}}} \right)} \right.
\\
&\qquad\qquad- \left. {\frac{1}{2}\int_{ - 1}^1 {\left( {f\left( s
\right) - \frac{{Af\left( { - 1/\sqrt 3 } \right) + Bf\left(
{1/\sqrt 3 } \right)}}{{A + B}}} \right)ds} } \right]^2 dt
\\
&= \int_{ - 1}^1 {\left[ {f\left( t \right) - \frac{1}{2}\int_{ -
1}^1 {f\left( s \right)ds} } \right]^2 dt}
\\
&= \int_{ - 1}^1 {f^2 \left( t \right)dt}  - \frac{1}{2}\left(
{\int_{ - 1}^1 {f\left( t \right)dt} } \right)^2
\\
&= 2\mathcal{T} \left( {f,f} \right),
\end{align*}
which gives by (\ref{eq2.16}) that
\begin{align*}
\left| {\mathcal{T}\left( {K,g'} \right) } \right| \le
\mathcal{T}^{1/2} \left( {f,f} \right)\mathcal{T}^{1/2} \left(
{g',g'} \right).
\end{align*}
Combining the above inequality with (\ref{eq2.15}) we get the
required result (\ref{eq2.14}).
\end{proof}

\centerline{}

\centerline{}

\end{document}